\documentclass[a4paper,10pt,reqno]{amsart}
\usepackage{xcolor}
\usepackage{graphicx}
\usepackage{amsmath}
\usepackage{amssymb}
\usepackage{comment}

\theoremstyle{plain}
\newtheorem{theorem}{Theorem}[section]

\newtheorem{lemma}[theorem]{Lemma}

\newtheorem{remark}[theorem]{Remark}

\newcommand{\dR}{\mathbb{R}}

\newcommand{\dZ}{\mathbb{Z}}

\let\altphi\phi
\let\phi\varphi
\let\varphi\altphi
\let\altphi\undefined

\begin{document}
	
	\title{Fukaya-Yamaguchi Conjecture in Dimension Four}
	
	\author{Elia Bru\`e, Aaron Naber, and Daniele Semola}
	
	\maketitle{}
	
	{\centering\footnotesize\it In honor of Gang Tian on his 65$^{th}$ Birthday.\par}
	
	\begin{abstract}
	Fukaya and Yamaguchi \cite{FukayaYamaguchi} conjectured that if $M^n$ is a manifold with nonnegative sectional curvature, then the fundamental group is uniformly virtually abelian.  In this short note we observe that the conjecture holds in dimensions up to four.
	\end{abstract}
	
\section{Statement of the main result}

\begin{theorem}\label{t:FY_4d}
	Let $M^n$ be a smooth manifold with nonnegative sectional curvature and $n\leq 4$.  Then there exists an abelian subgroup $A\leq \pi_1(M)$ of the fundamental group with universally bounded index $[\pi_1(M):A]\leq C(n)$. 
\end{theorem}

It is interesting to note that the above fails if one only assumes $\text{Ric}\geq 0$, see \cite{Wei_RicciNilpotent} and more recently \cite{BNS_Nilpotent} for examples in the closed case. In the case of $\text{Ric}\geq 0$ the fundamental group may not even be finitely generated, see for instance \cite{BNS_Milnor,BNS_Milnor6}.  However, it is unknown if the infinitely generated component must be abelian; for instance, it is unknown if there exists a normal abelian subgroup $A\leq \pi_1(M)$ such that the quotient $\pi_1(M)/A$ is finitely generated? \\

\section{Proof of the main result}

By the Cheeger-Gromoll soul theorem \cite{CheegerGromollsoul}, $M^n$ deformation retracts onto a compact totally geodesic submanifold. In particular, it is homotopically equivalent to a compact manifold with nonnegative sectional curvature of dimension $\leq n$. Hence we can assume without loss of generality that $M^n$ is compact.
\medskip

We break the proof down into two basic cases, which is whether or not the universal cover $\tilde{M}^n$ is compact or noncompact.  Let us first deal with the case when $\tilde M^n$ is compact, where in fact we will prove a slightly more general result about effective actions.  This generalization will prove useful in the noncompact context:

\begin{lemma}\label{l:compact}
	Let $(\tilde M^n,g)$ be a simply connected manifold with $n\leq 4$, $\sec\geq -1$ and $\mathrm{diam} \leq D$.  Then any finite group $\Gamma$ acting smoothly and effectively on $\tilde M$ admits an abelian subgroup $A\leq \Gamma$ which is generated by at most $C(n,D)$ elements and whose index is uniformly bounded $[\Gamma:A]\leq C(n,D)$ .
\end{lemma}
\begin{remark}
	Note that if $\sec\geq 0$ and $\tilde M^n$ is compact, then we may rescale in order to assume $\mathrm{diam}(\tilde M^n)\leq 1$ .  In particular, we have that $A$ is generated by at most $C(n)$ generators with $[\Gamma:A]\leq C(n)$ .
\end{remark}
\begin{proof}[Proof of Lemma \ref{l:compact}]

Let us first observe that in the case $n=2$ we have that $M^2=S^2$ and in the case $n=3$ we have that $M^3=S^3$ as they are simply connected closed manifolds\footnote{It will be enough that that $M^3$ is an integral homology sphere, so we do not really need to appeal to Perelman's proof of the Poincare conjecture.}.  In the case $n=4$, let us recall that the Euler characteristic of a simply connected four manifold $M^4$ satisfies
\begin{align}
	\chi(M^4) = 2+b_2\geq 2>0\, .
\end{align} 

In particular, a simply connected four manifold has positive Euler characteristic.  Let us now appeal to the results of Mundet i Riera \cite{Mundet19}.  In the cases where $M$ is either an integral homotopy sphere or has nonzero Euler characteristic, we have that $\text{Diff}(M)$ is a Jordan space.  More precisely and effectively, by \cite[Theorem 1.2]{Mundet19} we have that if $\Gamma$ is any smooth effective action on $M$ then there exists an abelian group $A\leq \Gamma$ and $C$ depending only on the dimensions of $M$ and $H^*(M,\dZ)$ such that: 
\begin{itemize}
\item[i)] $A$ is generated by at most $C$ elements; 
\item[ii)] $[\Gamma: A]\leq C$.  
\end{itemize}
If we combine with Gromov's betti number estimates \cite[Theorem 0.2B]{Gromovbetti}, which bounds for us the dimension of $H^*(M,\dZ)$, this finishes the proof of Lemma \ref{l:compact}.

Let us make the observation that in the case that $\Gamma$ is an oriented and free action on $M^4$, the result is even easier as one gets directly the order bound on $\Gamma$:
\begin{align}
	2|\Gamma| \leq \chi(M^4/\Gamma)|\Gamma| = \chi(M^4) = 2+b_2\, .
\end{align}
In the case $n=2$ or $n=3$ we may also have instead appealed to \cite[Theorem E]{DinkelbachLeeb} in order to make the required conclusions.
\end{proof}

With the compact case in hand let us now deal with the noncompact case and finish the proof of Theorem \ref{t:FY_4d}:

\begin{proof}[Proof of Theorem \ref{t:FY_4d}]
	We can apply Lemma \ref{l:compact} and assume that the universal cover $\tilde M^n$ is noncompact.  \\
	
{\bf Claim 1: $\tilde M^n = \dR^k\times N$ where $N$ is compact.}\\  

We first follow \cite{CheegerGromollRic} and apply the Toponogov splitting theorem \cite{Toponogov} in order to understand the structure of $\tilde M^n$.  So let us write the universal cover as an isometric splitting $\tilde M = \dR^k\times N$, where by assumption $k$ is maximal so that $\tilde N$ does not isometrically split any Euclidean factors. Note that in this context each isometry $\gamma\in \Gamma\equiv \pi_1(M)\leq \text{Isom}(\tilde M)$ splits $\gamma=\gamma_k\times \gamma_N$ where $\gamma_k\in \text{Isom}(\dR^k)$ and $\gamma_N\in \text{Isom}(N)$ .  The mappings $\rho_k: \text{Isom}(\tilde M)\to \text{Isom}(\dR^k)$ and $\rho_N: \text{Isom}(\tilde M)\to \text{Isom}(N)$ are clearly homomorphisms.
	
	So let us prove that $N$ is compact.  To prove the claim let $\hat M\subseteq \tilde M$ be a fundamental domain.  As $M$ is compact we have also that $\hat M$ is compactly supported.  Now if $N$ is not compact then we can find a sequence $x_j=(0,y_j)\in \dR^k\times N$ with $d(x_0,x_j)=d(y_0,y_j) = 2r_j\to \infty$ .  Let $x_j(t)=(0,y_j(t)):[-r_j,r_j]\to \tilde M$ be a geodesic between $x_0$ and $x_j$.  By the definition of the fundamental domain, we can act on $x_j(t)$ by a deck transformation $\gamma_j = \gamma^k_j\times \gamma^N_j\in \pi_1(M)$ so that the resulting minimizing geodesic $\sigma_j(t) = \gamma_j\cdot x_j(t)=(\gamma^k_j(0),\gamma^N_j\circ x_j(t))=(\sigma^k_j, \sigma^N_j(t)):[-r_j,r_j]\to \tilde M$ satisfies $\sigma_j(0)\in \hat M$, where $\sigma^k_j\in \dR^k$ and $\sigma^N_j(t)$ is a geodesic in $\tilde N$. As $\sigma_j(0)$ is compactly supported we can pass to a subsequential limit $\sigma_j\to \sigma = (\sigma^k,\sigma^N(t)):(-\infty,\infty)\to \tilde M$ to get a line $\sigma^N(t)$ in $\tilde N$.  We can now apply the Toponogov splitting \cite{Toponogov} to see that we have the isometric splitting $N = \dR\times N_1$, which is a contradiction to our assumption that $N$ did not split any Euclidean factors.  Thus $N$ is compact and this proves Claim 1. \\
	
	For the next step in the proof let us begin by considering the two groups $\Gamma_k = \rho_k(\Gamma)=\rho_k(\pi_1(M))\leq \text{Isom}(\dR^k)$ and $\Gamma^0_N=\ker \rho_k\leq \Gamma$.  Note that $\Gamma^0_N$ is normal in $\Gamma$, and by observing that $\Gamma^0_N$ fixes the $\dR^k$ factor we can naturally embed it $\Gamma^0_N\leq \text{Isom}(N)$. Note that as $\Gamma$ is a discrete cocompact action on $\tilde M$ and $N$ is compact, we have that $\Gamma_k$ is a discrete cocompact action on $\dR^k$.  It follows from the Bieberbach theorem that there exists an abelian group, in fact a lattice, $A_k\leq \Gamma_k$ with at most $K(k)$ generators for which $[\Gamma_k:A_k]\leq C(k)$.  
	
	
%
	
	We set $\Gamma'=\rho^{-1}_k(A_k)\leq \Gamma$, which we observe is a $C(k)$-finite index subgroup with normal subgroup $\Gamma^0_N\leq \Gamma'$ and at most $K=K(n)$ generators.  We have the short exact sequence
	
\begin{align}
	0\rightarrow \Gamma^0_N \rightarrow \Gamma'\rightarrow A_k\to 0\,\, .
\end{align}

As $A_k$ is abelian, we have that the commutator subgroup must satisfy $[\Gamma',\Gamma']\leq \Gamma^0_N$. Let us now define $\Gamma'_N:=\rho_N(\Gamma')\leq \text{Isom}(N)$ to be the image group with $\overline{\Gamma}'_N\leq \text{Isom}(N)$ its closure.  Note that $\overline{\Gamma}'_N$ is a compact Lie group. \\ 

{\bf Claim 2: The connected component of the identity of $\overline{\Gamma}'_N\leq \text{Isom}(N)$ is a torus $T^\ell\leq \overline{\Gamma}'_N$, and there exists a finite normal subgroup $\Sigma\leq \overline{\Gamma}'_N$ such that $T^\ell$ and $\Sigma$ generate $\overline \Gamma'_N$}.\\

To prove the claim let us first observe that $[\Gamma'_N,\Gamma'_N]\leq \Gamma^0_N\leq \text{Isom}(N)$, and hence as the commutators are continuous we have that $[\overline\Gamma'_N,\overline\Gamma'_N]\leq \Gamma^0_N\leq \text{Isom}(N)$.   In particular, $\overline\Gamma'_N/\Gamma^0_N$ is an abelian compact Lie group.  As it is abelian we have that its connected component is a torus and we can write $\overline\Gamma'_N/\Gamma^0_N = \hat T^\ell\times \hat \Gamma_A$, where $\hat\Gamma_A$ is itself a finite abelian group; the finiteness of $\hat\Gamma_A$ is due to compactness of $\text{Isom}(N)$.  The point to emphasize is that there exists in $\overline\Gamma'_N/\Gamma^0_N$ a finite abelian group $\{e\}\times\hat\Gamma_A$ with an element in each connected component of $\overline\Gamma'_N/\Gamma^0_N$.  By lifting $\{e\}\times\hat\Gamma_A$ to $\overline{\Gamma}'_N$ under the quotient map by $\Gamma^0_N$, we have the (normal, because $\hat\Gamma_A$ is normal) finite group $\Sigma  \leq \overline\Gamma'_N$, which satisfies the short exact sequence 

\begin{align}
	&0\rightarrow \Gamma^0_N \rightarrow \Sigma\rightarrow \hat\Gamma_A\to 0\,\, .
\end{align}
As $\hat\Gamma_A$ contains an element in each connected component of $\overline\Gamma'_N/\Gamma^0_N$, we have that $\Sigma$ contains an element in each connected component of $\overline\Gamma'_N$. As the identity component of $\overline\Gamma'_N/\Gamma^0_N$ is a torus $\hat T^\ell$, we necessarily have that the identity component of $\overline\Gamma'_N$ is a torus $T^\ell$, as claimed.\\

We have two final properties to check before we can finish the proof of the Theorem.  The first is to see that $\Sigma$ is virtually abelian, the second is to see that $T^\ell$ and $\Sigma$ commute:\\

{\bf Claim 3:  There exists an abelian subgroup $A_\Sigma\leq \Sigma$ with $K(n)$ generators and finite index $[\Sigma:A_\Sigma]\leq C(n)$}.\\

Observe that $\Sigma\leq \overline{\Gamma}'_N\leq \text{Isom}(N)$ is in particular a finite group with an effective action on $N$, and hence we may apply Lemma \ref{l:compact} to conclude the existence of $A_\Sigma\leq \Sigma$, as claimed.\\

{\bf Claim 4:  The subgroups $T^\ell$ and $\Sigma\leq \overline\Gamma'_N\leq \text{Isom}(N)$ commute.}\\

As $\Sigma$ is normal in $\overline\Gamma'_N$, we have for each $t\in T^\ell$ and $\sigma\in \Sigma$ that $t\sigma t^{-1}\in \Sigma$. By continuity considerations, as $\Sigma$ is finite and $T^\ell$ is connected, we then immediately get that $t\sigma t^{-1}=\sigma$. In particular, $T^\ell_N$ and $\Sigma$ commute, as claimed. \\

Let us now finish the proof of the Theorem.  Let $\bar A\leq \overline\Gamma'_N$ be the closed group generated by $T^\ell$ and $A_\Sigma\leq \Sigma$. Note by Claim 3 that $[\overline\Gamma'_N:\bar A]\leq C(n)$, and as $A_\Sigma$ and $T^\ell$ commute we have that $\bar A$ is abelian.  Let us define the homomorphism
\begin{align}
	\rho:\Gamma'\stackrel{\rho_N}{\longrightarrow}\Gamma'_N\to \overline\Gamma'_N\, ,
\end{align}
and then define
\begin{align}
	A:=\rho^{-1}(\bar A)\leq \Gamma'\, .
\end{align}
Note that $A$ is abelian as we can identify $A\leq \text{Isom}(\dR^k)\times \text{Isom}(N)$ and its projections to each factor are abelian.  Further, it follows that $[\Gamma': A]=[\overline\Gamma'_N:\bar A]\leq C(n)$.  When we combine this with the previous estimate $[\Gamma:\Gamma']\leq C(n)$ we get $[\Gamma: A]\leq C(n)$, which finishes the proof of the Theorem.
\end{proof}

\begin{remark}
In \cite[Section 6]{KPTAnn} the authors discuss an alternative strategy, due to Wilking, to perform the reduction to compact universal covers in the context of the Fukaya-Yamaguchi conjecture. Their approach relies on \cite[Corollary 6.3]{Wilking00}. 
\end{remark}

\end{document}